\documentclass[12pt, leqno]{amsart}
\usepackage[all]{xy}
\usepackage{amsfonts,amsmath,oldgerm,amssymb,amscd}
\UseComputerModernTips
 \usepackage{tabularx,ragged2e,booktabs,caption}

\newtheorem{theorem}{Theorem}

\newtheorem{lemma}[subsection]{{\bf Lemma}}

\newcommand{\al}{\alpha}

\newcommand{\Q}{\mbox{$\mathbb Q$}}

\begin{document}
\title{On the Galois group of Generalised Laguerre polynomials II } 
\author[Laishram]{Shanta Laishram}
 \address{Stat-Math Unit, India Statistical Institute\\
 7, S. J. S. Sansanwal Marg, New Delhi, 110016, India}
 \email{shanta@isid.ac.in}
\author[Nair]{Saranya G. Nair}
\address{Stat-Math Unit\\
Indian Statistical Institute, 8th Mile Mysore Road\\
 Bangalore, 560059}
 \email{saranya$ \_$vs@isibang.ac.in}
\author[Shorey]{T. N. Shorey}
\address{National Institute of Advanced Studies, IISc Campus\\
Bangalore, 560012}
\email{shorey@math.iitb.ac.in}
\begin{abstract} For real number $\alpha,$ Generalised Laguerre Polynomials (GLP) is a family of polynomials defined by
\begin{align*}
L_n^{(\alpha)}(x)=(-1)^n\displaystyle\sum_{j=0}^{n}\binom{n+\alpha}{n-j}\frac{(-x)^j}{j!}.
\end{align*}These orthogonal polynomials are extensively studied in Numerical Analysis and Mathematical Physics. In 1926, Schur initiated the study of algebraic properties of these polynomials. We consider the Galois group of Generalised Laguerre Polynomials $ L_n^{(\frac{1}{2}+u)}(x^2)$ when $u$ is a negative integer.

\end{abstract}
\maketitle
\pagenumbering{arabic}
\pagestyle{myheadings}
\markright{Galois group of Laguerre polynomials  $L_n^{(\frac{1}{2}+u)}(x^2)$ with  $-18 \leq u \leq -2$}
\markleft{Laishram, Nair and Shorey}

\section{{\bf Introduction}}
For real number $\alpha$ and integer $n \geq 1,$ the Generalised Laguerre Polynomials (GLP) is a family of polynomials defined by
\begin{align*}
L_n^{(\alpha)}(x)=(-1)^n\displaystyle\sum_{j=0}^{n}\binom{n+\alpha}{n-j}\frac{(-x)^j}{j!}.
\end{align*}

    These orthogonal polynomials have been extensively studied in various branches of analysis and mathematical physics where they play an important role. They are used in Gaussian quadrature to numerically compute integrals of the form $\displaystyle\int_{0}^{\infty}f(x)e^{-x}dx$. They satisfy second order linear differential equation
      \begin{align*}
      x y^{''}+ (\alpha+1-x)y^{'}+ny=0,\ y=L_n^{(\alpha)}(x)
      \end{align*}
      and the difference equation
      \begin{align*}
      L_n^{(\alpha)}(x)-L_n^{(\alpha-1)}(x)=L_{n-1}^{(\alpha)}(x).
      \end{align*}
      Schur \cite{Sch1}, \cite{Sch2} was the first to establish interesting and important algebraic properties of these polynomials.
      We define
              \begin{align*}
              \mathcal{L}_n^{(u)}(x)=\sum^n_{j=0}\binom{n}{j}(1+2(u+n))(1+2(u+n-1))\cdots (1+2(u+j+1))x^j
              \end{align*} and observe that $ \mathcal{L}_n^{(u)}(2x)=2^nn!L_n^{(\alpha)}(-x)$ and thus the irreducibility of $\mathcal{L}_n^{(u)}(x)$ implies irreducibility of ${L}_n^{(\alpha)}(x)$ and their assosciated Galois groups are same.
        Schur gave a formula for the discriminant $\Delta_n^{(\alpha)}$ of $ \mathcal{L}_n^{(\alpha)}(x)$ by \begin{align*}
            \Delta_n^{(\alpha)}=\displaystyle\prod_{j=1}^{n}j^j(\alpha+j)^{j-1}
            \end{align*}
            and calculated their assosciated Galois groups. Schur \cite{Sch1}, \cite{Sch2} showed that for every positive integer $n$, the polynomial $L_n^{(0)}(x)$ is irreducible and has associated Galois group the symmetric group $S_n$. He showed that $L_n^{(1)}(x)$ is irreducible for all positive integers $n$ and has associated Galois group the alternating group $A_n$ if
         $n > 1$ and $n$ is odd, and $S_n$ otherwise. Further, he showed that the polynomials $L_n^{(-1-n)}(x)$ have associated Galois group $A_n$ if $n \equiv 0 \ ({\rm mod}\ 4)$ and $S_n$ otherwise. Gow \cite{Gow} showed that if $n$ is an even integer $> 2$, then the Galois group associated with $L_n^{(n)}(x)$ is
         $A_n$ provided that the polynomial $L_n^{(n)}(x)$ is irreducible. Filaseta, Kidd and Trifonov \cite{Kidd} proved that for every integer $n > 2$ with $n \equiv 2 \ ({\rm mod}\ 4),$ the polynomial $L_n^{(n)}(x)$ is irreducible. For $n=2,$ $L_2^{(2)}(x)$ is reducible but its Galois group is $A_2=\{e\}.$
         These results settled the inverse Galois problem for $A_n$ explicitly that for every positive integer $n>1,$ there exists an explicit Laguerre polynomial of degree $n$ whose Galois group is the alternating group $A_n.$

        From now onwards, we always assume that
      \begin{align}\label{1new}
      \al=u+\frac{1}{2}
      \end{align}
      where $u$ is an integer . We recall that Hermite polynomials $H_{2n}(x)$ and $H_{2n+1}(x)$ are given by $$H_{2n}(x)=(-1)^n2^{2n}n!L_{n}^{(-\frac{1}{2})}(x^2) \ \text{and} \ H_{2n+1}(x)=(-1)^n2^{2n+1}n!x L_{n}^{(\frac{1}{2})}(x^2).$$
        Schur \cite{Sch1}, \cite{Sch2}, proved that $L_n^{(-\frac{1}{2})}(x^2)$ and $L_n^{(\frac{1}{2})}(x^2)$ are irreducible and these imply the irreducibility of $H_{2n}(x)$ and $H_{2n+1}(x)/x.$ We observe that $u \in \{-1,0\}$ in these results and Laishram \cite{la15hardy} proved that the Galois group is $S_n$ when $u \in \{-1,0\}.$ Laishram, Nair and Shorey \cite{LaNaSh15},\cite{LaNaSh15b} showed that $L_n^{(\alpha)}(x)$ with $\alpha=u+\frac{1}{2}$
       and $1 \leq u \leq 45$ are irreducible except when $(u,n)=(10,3)$ and have associated Galois group $S_n$ other than in the case of $(u,n)=(10,3)$ where the Galois group is $\mathbb{Z}_2$. Further Nair and Shorey \cite{NaSh} proved that $L_n^{(\alpha)}(x)$ with $-18 \leq u \leq -2$ are irreducible. In this paper, we compute the Galois groups of $L_n^{(\alpha)}(x)$ with $-18 \leq u \leq -2$.
       \begin{theorem}\label{lag} $(n,u)=$
       Let  $\alpha=u+\frac{1}{2}$ and $-18 \leq u\leq -2$. Then the Galois group of $L^{(\alpha)}_n(x)$
                    is $S_n$ except when $(n,u) \in \{(8,-5),(8,-6),(9,-5),(9,-6)
                    (16,-9),(16,-10),\\(24,-13),(24,-14),(25,-13),(25,-14), (32,-17), (32,-18)\}$ in which cases the Galois group is $A_n$.
       \end{theorem}

        \section{Preliminaries}
       We will use a result of Hajir \cite{hajir} which gives a criterion for an irreducible polynomial to have large Galois group using Newton polygons. We restate the result which is \cite[Lemma 3.1]{hajir}. For an integer $x,$ let $\nu(x)=\nu_p(x)$ be the highest power of $p$ dividing $x$ and we write $\nu(0)=\infty.$

       \begin{lemma} \label{An}
       Let $f(x)=\sum^m_{j=0}\binom{m}{j}c_jx^j\in \Q[X]$ be an irreducible polynomial of degree $m$. Let $p$ be a
       prime with $\frac{m}{2}< p<m-2$ such that
       \begin{itemize}
       \item[$(i)$] $\nu_p(c_j) \geq 0 $ for $j=0,1,\ldots,m,$
       \item[$(ii)$] $\nu_p(c_0)=1$,
       \item[$(iii)$] $\nu_p(c_j)\ge 1$ for $0\le j\le m-p$,
       \item[$(iv)$] $\nu_p(c_p)=0$.
       \end{itemize}
       Then the Galois group of $f$ contains $A_m.$ Further Galois group is
       $A_m$ if discriminant of $f\in \Q^{*2}$ and $S_m$ otherwise.
       \end{lemma}
       Next, we state a result which we deduce from a result due to Harborth and Kemnitz and state it as a lemma so that we can use it easily.
        \begin{lemma}\label{exist of p from sell}
         There exists a prime $p$ satisfying $\frac{2}{3}n \leq p <n-2 \ \text{for} \  n \geq 14.$
               \end{lemma}
               \begin{proof}
               In \cite{har}, Harborth and Kemnitz proved that there exists a prime $p$ satisfying
                               \begin{align*}
                               x <p <\frac{6}{5}x \ \text{for} \  x \geq 25.
                               \end{align*}
              We take $x=\frac{3}{4}n.$ Then for $n \geq 34$, we conclude that there exists a prime $p$ in  $ (\frac{3}{4}n,\frac{9}{10}n) \subset [\frac{2}{3}n,n-2).$ For $14 \leq n \leq 33,$ we check that $[\frac{2}{3}n,n-2)$ contains a prime.
    \end{proof}
        \begin{lemma}\label{prime cong}
         Let $r \in \{1,3\}.$ The interval $(x,1.048x]$ contain  primes congruent to $r$  modulo $4$ when
         $x \geq 887.$
         \end{lemma}
         This follows from  Cullinan and Hajir \cite[Theorem 1]{CuHa} with $k=4.$
         \qed

          \begin{lemma}\label{Laguer}
          For $-18 \leq u \leq -2,$ the polynomials $L_n^{(\frac{1}{2}+u)}(x)$ are irreducible.
          \end{lemma}
          This follows from Nair and Shorey \cite[Theorem 1]{NaSh}.
      \section{Galois group of $L_n^{(\alpha)}(x)$ }

       The results in this section are more general than required for the proof of Theorem \ref{lag}. For a fixed $u,$ if $n \leq 13$ we can compute Galois group directly using MAGMA. Thus we assume $n \geq 14$ and we always write
      \begin{align*}
      v=-u.
      \end{align*} Now we state our main lemma.
      \begin{lemma}\label{main}
      Assume $\mathcal{L}_n^{(u)}(x)$ is irreducible and $n \geq \max\{14,2v-1\}$. If there exists a prime $\in (2v-3,n-2)$, then the Galois group of $\mathcal{L}_n^{(u)}(x)$ contains $A_n.$
      \end{lemma}
      \begin{proof}
    By Lemma \ref{exist of p from sell}, there exists a prime $\in [\frac{2}{3}n,n-2)$. So we choose $p \in (2v-3,n-2)$ such that
      \begin{align}\label{p>}
      p \geq \max\left\{\frac{2}{3}n,2v-1\right\}.
      \end{align}
       We write
            \begin{align*}
            c_j=\displaystyle\prod_{i=j+1}^{n}(1+2(u+i))=\displaystyle\prod_{i=j+1}^{n}(1-2(v-i)),
            \end{align*}
       $p=2v+l$ and $l \geq -1.$ We observe that $p< 1+2(p-v+1)$ since $p > 2v-3.$ Further, $4v+\frac{3}{2}(l-1)=\frac{3}{2}(2v+l)+v-\frac{3}{2}=\frac{3}{2}p+v-\frac{3}{2} > n$ since $p \geq \frac{2}{3}n$ and $v \geq 2.$ Thus, it follows that
      \begin{align}\label{3p>}
     3p > 2(n-v)+3.
      \end{align}Therefore, we conclude that $p \nmid c_p.$

      Next we show that $\nu_p(c_j) \geq \nu_p(c_0) =1$ for $1 \leq j \leq n-p.$ We have $c_0=\pm 1 \cdot 3 \cdot (2v-3)\cdot 1\cdot 3 \cdot (1+2(n-v))$ and $c_{n-p}=(1+2(n+1-p-v))\cdots (1+2(n-v))$. Since $n \geq 2v-1$, we get
      \begin{align}\label{p<}
      p <n \leq 2(n-v)+1.
      \end{align} This together with $p >2v-3$ and \eqref{3p>}, it follows that $\nu_p(c_0)=1.$ Now it suffices to show that $p|c_{n-p}.$ If $n+1 <p+v$, then $1+2(n-v-p+1) <1$ which together with \eqref{p<}  imply $p|c_{n-p}.$ So, we assume that $n+1 \geq p+v$ and in this case $c_{n-p}$ is the product of $p$ consecutive odd numbers and hence divisible by $p.$ Therefore we conclude that the Galois group of
      $\mathcal{L}_n^{u}(x)$ contains $A_n$ by Lemma \ref{An}.
      \end{proof}
      We recall that the discriminant $\Delta_n^{(u)}$ of $\mathcal{L}_n^{(u)}(x)$  is
      given by
                     \begin{align*}
                     \Delta^{(u)}_n:=\displaystyle\prod^n_{j=2}j^j(\frac{2u+1+2j}{2})^{j-1}.
                     \end{align*}
      \begin{lemma}\label{disc}
       Let $2n \geq \max\{28,2v-5\}$ and $n \notin  \{17,18,26,27\}$. Then, $\Delta^{(u)}_n$ is not a square for all pairs of $(n,u)$ satisfying above conditions.
      \end{lemma}

      \begin{proof}
       Given a positive integer $n,$ let $n_o$ and $n_e$ denote the largest odd and even number less than or equal $n$ respectively. Then \begin{align*}
               \Delta^{(u)}_n= 1 \cdot 3 \cdot 5 \cdots n_o \cdot (2u+1+4)(2u+1+8)\cdots(2u+1+2n_e)\times 2^{\frac{-n_e}{2}}\times \square
               \end{align*}
               where $\square$ denotes a term in $\mathbb{Q}^{*2}.$ If $n \equiv 2, 3 \ (\text{mod} \ 4),$ then $\frac{n_e}{2}$ is odd and hence $\Delta^{(u)}_n$ is not a square. Hence $n \equiv 0,1 \ (\text{mod} \ 4).$ We consider \begin{align*}
               D_v= 1 \cdot 3 \cdot 5 \cdots n_o(-2v+1+4)(-2v+1+8)\cdots(-2v+1+2n_e).
               \end{align*}
               It suffices to show that $D_v$ is not a square. Let $r \in \{1,3\}$ be such that $-2v+1 \equiv r \ (\text{mod} \ 4).$ Then
               \begin{align*}
               D_v=1 \cdot 3 \cdots n_o (-2v+1+4)(-2v+1+8)\cdots (r-4) \cdot r \cdot (r+4)\cdots (-2v+1+2n_e).
               \end{align*}
               Let $p_0$ be the largest prime with $\frac{2}{3}n \leq p <n$ and $p \equiv 3r \ (\text{mod} \ 4)$. We take $x=\frac{2}{3}n$ and $n > 1330$ in Lemma \ref{prime cong} to conclude that the interval $[\frac{2}{3}n,n-2)$ contain both primes congruent to $1$ and $3$ modulo $4$. We check that for $14 \leq n \leq 1330,$ the interval $[\frac{2}{3}n,n-2)$ contain both primes congruent to $1$ and $3$ modulo
               $4$ except for $n \in \{17,18,26,27\}$. Let $n \notin  \{17,18,26,27\}$. Then $p_0 || 1 \cdot 3 \cdot 5 \cdots n_o$. Further $3p_0  >2n >2n_e-2v+1$ and $2v-5 \leq 2n <3p_0$. Since $p \equiv 3r \ (\text{mod} \ 4)$ and $-2v+1 \equiv r \ (\text{mod} \ 4)$, the above conditions imply that $p_0|| D_v$ and hence $D_v$ is not a square. Therefore $\Delta^{(u)}_n$ is not a square.
      \end{proof}
       The following result is obtained by direct computation using MAGMA.
      \begin{lemma}\label{comp}
       Let $2\leq n \leq 40$ and $-18 \leq u\leq -2$. Then the Galois group of $\mathcal{L}^{(\alpha)}_n(x)$
                          is $S_n$ except when $(n,u) \in \{(8,-5),(8,-6),(9,-5),(9,-6)
                          (16,-9),(16,-10),(24,-13),\\(24,-14),(25,-13),(25,-14), (32,-17), (32,-18)\}$ in which cases the Galois group is $A_n$.
      \end{lemma}

We now prove Theorem \ref{lag}. 

      \subsection*{Proof of Theorem \ref{lag}:}
      Let $-18 \leq u\leq -2$. Then $\mathcal{L}_n^{(u)}(x)$ are irreducible by Lemma \ref{Laguer}. Further, by Lemma \ref{comp}, we can always assume that $n \geq 41.$ Since $v \leq 18,$ we have $n \geq 2v-1.$ If $n \geq 50,$ then $2v-3 < \frac{2}{3}n$ and hence using Lemma \ref{exist of p from sell}, we conclude that there exists a prime in $(2v-3,n-2).$ If $41 \leq n \leq 49,$ we directly check that such a prime exists as $2v-3 \leq 33.$ Hence, by Lemmas \ref{main} and \ref{disc}, the Galois group of $\mathcal{L}_n^{(u)}(x)$ is $S_n.$
\qed

\section*{Acknowledgements}
The first named author acknowledges the support of the MATRICS Grant of SERB, DST, India.

      \end{document}